\documentclass[sn-mathphys]{sn-jnl}% Math and Physical Sciences Reference Style
\usepackage{amsmath, amsthm, amsfonts, amssymb, thmtools}
\usepackage{hyperref}
\jyear{2022}%

\theoremstyle{thmstyleone}%
\newtheorem{theorem}{Theorem}%  meant for continuous numbers
\newtheorem{lemma}[theorem]{Lemma}

\numberwithin{equation}{section}

\theoremstyle{thmstyletwo}%
\newtheorem{example}{Example}%
\newtheorem{remark}{Remark}%

\theoremstyle{thmstylethree}%

\newcommand{\mD}{\ensuremath{\mathbb{D}}}

\begin{document}

\title[Normality and partial sharing of sets with differential polynomials]{Normality through partial sharing of sets with differential polynomials}


\author[1]{\fnm{Kuldeep Singh} \sur{Charak}}\email{kscharak7@rediffmail.com}

\author*[1]{\fnm{Nikhil} \sur{Bharti}}\email{nikhilbharti94@gmail.com}

\author[2]{\fnm{Anil} \sur{Singh}}\email{anilmanhasfeb90@gmail.com}

\affil[1]{\orgdiv{Department of Mathematics}, \orgname{University of Jammu}, \orgaddress{\city{Jammu}, \postcode{180006}, \state{Jammu and Kashmir}, \country{India}}}

\affil[2]{\orgdiv{Department of Mathematics}, \orgname{Maulana Azad Memorial College}, \orgaddress{\city{Jammu}, \postcode{180006}, \state{Jammu and Kashmir}, \country{India}}}

%\affil[3]{\orgdiv{Department}, \orgname{Organization}, \orgaddress{\street{Street}, \city{City}, \postcode{610101}, \state{State}, \country{Country}}}


\abstract{This article aims at finding  sufficient conditions for a family of meromorphic functions to be normal by involving partial sharing of sets with differential polynomials. Moreover, corresponding results for normal meromorphic functions are also established which improve and generalize many known results.}




\maketitle

\section{Introduction and Statement of Results}
%We set the following notations throughout the paper:
%\begin{itemize}
%\item $\mathcal{H}(D):$ the class of all holomorphic %functions on the domain $D\subseteq\mathbb{C}$
%\item $\mathcal{M}(D):$ the class of all meromorphic functions on the domain $D\subseteq\mathbb{C}$
%\item $\mathbb{C}_{\infty}:$ the extended complex plane
%\item $\mathbb{D}:$ the open unit disk in $\mathbb{C}.$ 
%\item $D(a,r):$ the disk with center $a$ and radius $r.$
%\end{itemize}
A family $\mathcal{F}$ of meromorphic functions in a domain $D\subset \mathbb{C}$ is said to be {\it normal} in $D$ if every sequence of functions in $\mathcal{F}$ contains  a subsequence which converges locally uniformly in $D$ with respect to the spherical metric to a function which is either meromorphic in $D$ or identically $\infty.$ Normality of a family of holomorphic functions in $D$ is defined analogously with respect to the Euclidean metric (see \cite{schiff, zalcman-1, zalcman-2}). On the other hand, a meromorphic function $f$ in the open unit disk $\mathbb{D}$ is said to be normal  if the family $\mathcal{G}=\left\{f\circ \psi: \psi\in \mathcal{A}(\mathbb{D})\right\}$ forms a normal family in $\mathbb{D},$ where $\mathcal{A}(\mathbb{D})$ is the group of conformal automorphisms of $\mathbb{D}.$ The study of normal functions was initiated implicitly by Noshiro \cite{noshiro} in $1939.$ Subsequently, Lehto and Virtanen \cite{lehto} extended the definition of normal meromorphic functions in $\mathbb{D}$ to arbitrary simply connected domains. Following the work of Yoshida \cite{yoshida}, Noshiro \cite[Theorem 1]{noshiro} gave a fundamental  criterion for normal functions in terms of growth of their spherical derivatives, namely that {\it a meromorphic function $f$ in the open unit disk $\mathbb{D}$ is normal if and only if $\sup\limits_{z\in\mathbb{D}}(1-\lvert z\rvert^2)f^{\#}(z)<\infty.$} Owing to its immense pertinency in geometric function theory, particularly in analyzing the boundary behaviour of a meromorphic function, many authors have explored properties of normal meromorphic functions from the geometric as well as the analytic point of view (see, for example \cite{anderson, pommerenke, yamashita}). 

\medskip

For the sake of convenience, we shall denote by $\mathcal{H}(D)$ and $\mathcal{M}(D)$ respectively the classes of all holomorphic and meromorphic functions in a domain $D\subseteq\mathbb{C},$ $\mathbb{C}_{\infty}=\mathbb{C}\cup\left\{\infty\right\}$ shall denote the extended complex plane, by $\mathbb{D}$ we shall denote the open unit disk and finally $D(a,r)$ shall denote the open disk with center $a$ and radius $r.$

\medskip

Let $f,g\in\mathcal{M}(D)$ and $S$ be a subset of $\mathbb{C}.$ Then we say that $f$ and $g$ share the set $S$ in $D$ if $E(f, S)=E(g, S),$ where $$E(f, S):=\bigcup\limits_{s\in S}\left\{z\in D: f(z)=s,~\mbox{ignoring multiplicity}\right\}.$$

More generally, for any $S_1, S_2\subseteq\mathbb{C}$ and $f,g\in\mathcal{M}(D),$ the pair $(f,g)$ is said to share the pair $(S_1, S_2)$ if $E(f, S_1)=E(g, S_2).$ In this case, we write $f(z)\in S_1\Leftrightarrow g(z)\in S_2.$ However, if either $E(f, S_1)\subseteq E(g, S_2)$ or $E(f, S_1)\supseteq E(g, S_2),$ then we say that $(f,g)$ share $(S_1, S_2)$ partially and we write $f(z)\in S_1\Rightarrow (\mbox{or} \Leftarrow)~g(z)\in S_2.$ In particular, if $f(z)=a\Leftrightarrow g(z)=a$ (respectively, $f(z)=a\Rightarrow g(z)=a$) for some value $a\in\mathbb{C}_{\infty},$ then $a$ is said to be a shared value (respectively, partially shared value) of $f$ and $g$ in $D.$

\medskip

Our objectives here are to obtain normality criteria for a family of meromorphic functions by involving partial sharing of sets with differential polynomials, and to find the corresponding criteria for normal meromorphic functions. Before we state our main results, we give some necessary background.

\medskip

Although the correspondence between normality and shared values was long established in $1992$ by Schwick (see \cite[Theorem 1]{schwick}), yet the first connection between normality and shared sets was established by Liu and Pang \cite[Theorem 1]{liu-pang} in $2007.$ Precisely, they proved:

\begin{theorem}
Let $\mathcal{F}\subset\mathcal{M}(D)$ and $S=\left\{a_1, a_2, a_3\right\}$ be a set in $\mathbb{C}.$ If for each $f\in\mathcal{F},$ $f(z)\in S\Leftrightarrow f'(z)\in ,$ then $\mathcal{F}$ is normal in $D.$
\end{theorem}

Other results related to normal families and shared sets can be found in \cite{cai, chang-wang, fang-zalcman, li}. It is noteworthy to mention that due to the close relationship between normal families and normal functions, one anticipates a criterion for normal functions corresponding to a known criterion of normal families and vice versa. However, this correlation does not always hold, for instance see \cite[p. 193]{hayman-storvick}.

\medskip

In $2016,$ Xu and Qiu \cite[Theorem 1.2]{xu-qiu} considered the analogous problem of sharing of sets with derivatives for normal functions and obtained the following:

\begin{theorem}
Let $f\in\mathcal{M}(\mathbb{D})$ and let $S_1=\left\{a_1, a_2, a_3\right\}$ and $S_2=\left\{b_1,b_2,b_3\right\}$ be two set in $\mathbb{C}.$ If $f(z)\in S_1\Leftrightarrow f'(z)\in S_2$ in $\mathbb{D},$ then $f$ is a normal function.
\end{theorem} 

In $2019,$ Chen and Tong \cite[Theorem 1]{chen-tong} considered sharing of sets with higher derivatives and proved the following criterion for normal functions:

\begin{theorem}\label{thm:tong}
Let $f\in\mathcal{M}(\mathbb{D}),$ $S_1=\left\{a_1, a_2, a_3\right\}$ and  $S_2=\left\{b_1,b_2,b_3\right\}$ be two finite subsets of $\mathbb{C}$ and $k\in\mathbb{Z}^{+}.$ If $$f(z)\in S_1\Leftrightarrow f^{(k)}(z)\in S_2,$$ and $$\max\limits_{0\leq i\leq k-1}\lvert f^{(i)}(z)\rvert=0 \mbox{ whenever }~f(z)\in S_1$$ hold in $\mathbb{D},$ then $f$ is a normal function.
\end{theorem} 

In $2020,$ Cai et al. \cite[Theorem 2]{cai} obtained normality criterion for families of meromorphic functions corresponding to Theorem \ref{thm:tong} as:

\begin{theorem}\label{thm:cai}
Let $\mathcal{F}\subset\mathcal{M}(D),$ $S_1=\left\{a_1, a_2, a_3\right\}$ and $S_2=\left\{b_1,b_2,b_3\right\}$ be two finite subsets of $\mathbb{C}.$ Let $k$ and $m$ be two positive integers and suppose that for each $f\in\mathcal{F}$ and $a\in S_1,$ $f-a$ has zeros of multiplicity at least $k.$ If $$f(z)\in S_1\Leftrightarrow \left(f^{(k)}\right)^m\in S_2,$$ then $\mathcal{F}$ is normal in $D.$ 
\end{theorem}

Recently Singh and Lal \cite[Theorem 1]{vb} improved upon Theorem \ref{thm:tong} by considering partial sharing of sets in the following manner:

\begin{theorem}\label{thm:vb}
Let $f\in\mathcal{M}(\mathbb{D}),$ $S_1$ and $S_2$ be any two finite sets in $\mathbb{C}$ with $\#(S_1)\geq 3$ and $k\in\mathbb{Z}^{+}.$ If $$f^{(k)}(z)\in S_1\Rightarrow f(z)\in S_2,$$ and $$\max\limits_{1\leq i\leq k}\lvert f^{(i)}(z)\rvert=0 \mbox{ whenever }~f(z)\in S_1$$ in $\mathbb{D},$ then $f$ is a normal function.
\end{theorem} 
 Note that the condition ``$\max\limits_{1\leq i\leq k}\lvert f^{(i)}(z)\rvert=0 \mbox{ whenever }~f(z)\in S_1$" in Theorem \ref{thm:vb} is equivalent to the condition that for each $a\in S_1,$ $f-a$ has zeros of multiplicity at least $k+1.$ 

\medskip

Regarding Theorem \ref{thm:cai} and Theorem \ref{thm:vb}, we consider the following questions:

\begin{itemize}
	\item[Q.$1$] Does the conclusion of Theorem \ref{thm:cai} remain valid under the weaker hypothesis of partial sharing of sets between $f^{(k)}$ and $f$? If yes, then does the conclusion still remain valid if $f^{(k)}$ is replaced by a  differential polynomial of $f?$
	\item[Q.$2$] Does the conclusion of Theorem \ref{thm:vb} remain valid if $f^{(k)}$ is replaced by a differential polynomial of $f?$
		\item[Q.$3$] Does there exist a normality criterion for a function $f\in\mathcal{M}(D)$ under the hypothesis of Theorem \ref{thm:cai}?
	\item[Q.$4$] Does there exist a normality criterion for a family $\mathcal{F}\subset\mathcal{M}(D)$ under the hypothesis of Theorem \ref{thm:vb}?
\end{itemize}

To answer these questions, we need some preparation:

\medskip

Let $k$ be a positive integer and let $n_0, n_1,\ldots, n_k$ be non-negative integers (not all zeros). A mapping $M: \mathcal{M}(D)\rightarrow\mathcal{M}(D)$ given by $$M[f]=b\cdot\prod\limits_{j=0}^{k}\left(f^{(j)}\right)^{n_j}~\mbox{for all}~f\in\mathcal{M}(D),$$ where $a\in\mathcal{M}(D),~b\not\equiv 0,$ is called a differential monomial of degree, $d(M):=\sum\limits_{j=0}^{k}n_j$ and weight, $w(M):=\sum\limits_{j=0}^{k}(j+1)n_j.$ 

We call $b,$ the coefficient of $M$ and if $b\equiv 1,$ then $M$ is said to be normalized differential monomial. Also, the number $k$ is known as the differential order of $M.$

Let $f\in\mathcal{M}(D)$ and $$M_i[f]:=b_i\prod\limits_{j=0}^{k_i}\left(f^{(j)}\right)^{n_j},$$ $1\leq i\leq m,$ be $m$ differential monomials. Then the sum $$P:= \sum\limits_{i=1}^{m}M_i$$ is called a differential polynomial and the quantities $d(P):=\max\left\{d(M_i): 1\leq i\leq m\right\}$ and $w(P):=\max\left\{w(M_i):1\leq i\leq m\right\}$ are called the degree and weight of the differential polynomial $P$ respectively. Also, we call the number $\kappa:=\max\left\{k_i:1\leq i\leq m\right\},$ the differential order of $P.$ If $d(M_1)=d(M_2)=\cdots=d(M_m),$ then $P$ is said to be a homogeneous differential polynomial.

\medskip

In the present paper, we are concerned with the differential polynomials of the form 
\begin{equation}\label{eqn:1}
P=\sum\limits_{i=1}^{m}a_i~M_i,
\end{equation}
where $$M_i[f]=\prod\limits_{j=1}^{k_i}\left(f^{(j)}\right)^{n_j},1\leq i\leq m~$$ are normalized differential monomials and the coefficients $a_{i}$ are holomorphic in $D.$

Moreover, we set 
\begin{eqnarray}\label{eqn:2}
1+\alpha:=\frac{w(M_1)}{d(M_1)}\geq\frac{w(M_t)}{d(M_t)}, ~\mbox{ for }~ 2\leq t\leq m.
\end{eqnarray}
Note that the arrangement on the right hand side of \eqref{eqn:2} occurs naturally.

Furthermore, we assume that the coefficients $a_i$ are non-vanishing for those $M_i$ for which $$\frac{w(M_i)}{d(M_i)}=1+\alpha~\mbox{and}~\alpha\geq\kappa~\mbox{is an integer}.$$ 

\medskip

Finally, we assume that the reader is familiar with the standard notations of the Nevanlinna Theory like $m(r,f),~N(r,f),~T(r,f)$ (see \cite{hayman}).

\medskip

Now we state our main results: 

\begin{theorem}\label{thm:1}
Let $\mathcal{F}\subset\mathcal{M}(D)$ and $S_1$ and $S_2$ be two finite subsets of $\mathbb{C}$ with $\#(S_1)\geq 3.$ Let $P$ be a differential polynomial defined in \eqref{eqn:1} and satisfying \eqref{eqn:2}. Suppose that for each $f\in\mathcal{F}$ and $a\in S_1,$ $f-a$ has zeros of multiplicity at least $\alpha+1.$ If $$P\left[f\right](z)\in S_1\Rightarrow f(z)\in S_2$$ in $D,$ then $\mathcal{F}$ is normal in $D.$
\end{theorem} 

\begin{example}
Let $\mathcal{F}=\left\{f_n:n\geq 2\right\}$ be a family of meromorphic functions in the punctured open unit disk  $\mD ^*:=\left\{z\in\mathbb{C}:0<\lvert z\rvert <1\right\}$ given by $$f_n(z)=\frac{nz^2}{2}+n.$$ Then it is easy to see that $\lvert f_n(z)\rvert\geq 1$ in $\mD^*.$ Let $S_1=\left\{a_1, a_2, a_3\right\},$ where $a_i$'s are distinct complex numbers such that $\lvert a_i\rvert <1~(i=1,2,3).$ Clearly, for each $a\in S_1,$ $f_n-a$ has zeros of arbitrary multiplicity.

Let $$P[f_n]= f_n'+\frac{1}{z^2}\cdot (f_n')^2.$$ Then $P[f_n](z)=nz+n^2$ and so $P[f_n](z)\in S_1\Rightarrow f_n(z)\in S_2,$ for any finite set $S_2$ in $\mathbb{C}.$ Obviously, the family $\mathcal{F}$ is normal in $\mD^*.$ This illustrates Theorem \ref{thm:1}. \end{example}

Theorem \ref{thm:1} gives affirmative answers to questions Q.$1$ and Q.$4.$

\begin{remark}
If $\mathcal{F}$ happens to be a family of holomorphic functions in $D$ and $0\notin S_1,$ then the condition $\#(S_1)\geq 3$ in Theorem \ref{thm:1} can be replaced by the condition $\#(S_1)\geq2.$ This is a direct consequence of Picard's Theorem. In particular, if the differential polynomial $P$ in Theorem \ref{thm:1} is replaced by some ordinary derivative $\left(f^{(k)}\right)^m$ with $\alpha$ replaced by $k,$ where $k,~m$ are positive integers, then the conclusion of Theorem \ref{thm:1} remains valid even if $0\in S_1$ by \cite[Lemma 4]{chen-fang}. However, the cardinality of $S_1$ cannot be reduced to one as the following example demonstrates:
\end{remark}

\begin{example}
Let $k$ and $m$ be any two positive integers and $\mathcal{F}=\left\{f_n: n\in\mathbb{N}\right\}$ be a family of holomorphic functions in $\mathbb{D}$ given by $$f_n(z)=e^{nz}.$$ Let $P[f_n]=\left(f_n^{(k)}\right)^m$ and $S_1=\left\{0\right\}.$ Since each $f_n$ omits $0,$ it follows that $f_n$ has zeros of arbitrary multiplicity. Obviously, $P[f_n](z)\in S_1\Rightarrow f_n(z)\in S_2,$ for any finite set $S_2$ in $\mathbb{C}.$ However, the family $\mathcal{F}$ is not normal in $\mathbb{D}.$
\end{example}

%\begin{example}
%Let $k$ and $m$ be any two %positive integers and %$\mathcal{F}=\left\{f_n: %n\in\mathbb{N}\right\}$ be a %family of meromorphic functions %in $\mathbb{D}$ given by %$$f_n(z)=e^{nz}.$$ Let %$P[f_n]=\left(f_n^{(k)}\right)^m%$ and $S_1=\left\{0\right\}.$ %Clearly, $f_n$ has zeros of %arbitrary multiplicity and %$P[f_n](z)\in S_1\Rightarrow %f_n(z)\in S_2,$ for any finite %set $S_2$ in $\mathbb{C}.$ %However, the family %$\mathcal{F}$ is not normal in %$\mathbb{D}.$
%\end{example}

A criterion for normal meromorphic functions corresponding to Theorem \ref{thm:1} is

\begin{theorem}\label{thm:2}
Let $f\in\mathcal{M}(\mathbb{D})$ and $S_1$ and $S_2$ be two finite subsets of $\mathbb{C}$ with $\#(S_1)\geq 3.$ Let $P$ be a differential polynomial defined in \eqref{eqn:1} and satisfying \eqref{eqn:2}. Suppose that for each $a\in S_1,$ $f-a$ has zeros of multiplicity at least $\alpha+1.$ If $$P\left[f\right](z)\in S_1\Rightarrow f(z)\in S_2$$ in $\mathbb{D},$ then $f$ is a normal function.
\end{theorem}

Theorem \ref{thm:2} gives affirmative answer to question Q.$2.$ Also, one can obtain an affirmative answer to question Q.$3$ by a simple modification of Theorem \ref{thm:2} with $S_1 \mbox{ and }S_2$ as three point sets and by considering the sharing as $f(z)\in S_1\Leftrightarrow \left(f^{(k)}\right)^m\in S_2$ under the assumption that for each $a\in S_1,$ $f-a$ has zeros of multiplicity at least $k.$

\begin{remark}
It is important to note that if the cardinality of the set $S_1$ in Theorem \ref{thm:1} and Theorem \ref{thm:2} is greater than $4$, then the condition ``for each $a\in S_1,$ $f-a$ has zeros of multiplicity at least $\alpha+1$" is enough to ensure normality. This follows immediately by a simple application of Lemma \ref{lem:2} and Lemma \ref{lem:1}, respectively, together with the fact that a non-constant meromorphic function cannot have more than four totally ramified values. Furthermore, the set $S_1$ in Theorem \ref{thm:1} and Theorem \ref{thm:2} can be taken to be a two point set consisting of  distinct non-zero complex values of which at least one is not a Picard exceptional value. Also, it easily follows from \cite[Theorem 5]{grahl} that the two point set $S_1$ may contain zero if the differential polynomial $P$ is of the form $P=\sum\limits_{i=1}^{m}a_i M_i$ with non-zero constant coefficients $a_i$ and normalized differential monomials $M_i$ in $f'$ satisfying at least one of the following conditions:
\begin{itemize}
    \item[$(i)$] $d(M_i)\geq 2$ for all $i=1,2,\ldots,m;$
    \item [$(ii)$] $w(M_1)\geq w(M_j)+2$ for all $j=2,\ldots,m.$ 
\end{itemize}
\end{remark}

The following example shows that the condition ``for each $f\in\mathcal{F}$ and $a\in S_1,$ $f-a$ has zeros of multiplicity at least $\alpha+1$" in Theorem \ref{thm:1} cannot be dropped.

\begin{example}
Let $\mathcal{F}=\left\{f_n: f_n(z)=nz,~n\in\mathbb{N}\right\}$ be a family of  meromorphic functions in $\mathbb{D}$ and let $S_1$ be any three point set in $\mathbb{C}\setminus\mathbb{N}.$ Then for each $a\in S_1,$ $f_n-a$ has only simple zeros. Consider $P[f_n]=f_n'.$ Then $P[f_n](z)=n.$ Clearly, $$P[f_n](z)\in S_1\Rightarrow f_n(z)\in S_2$$ for any finite set $S_2$ in $\mathbb{C}.$  However, the family $\mathcal{F}$ is not normal in $\mathbb{D}.$
\end{example}

The next example demonstrates that the condition ``$P\left[f\right](z)\in S_1\Rightarrow f(z)\in S_2$" in Theorem \ref{thm:1} is essential.

\begin{example}
Let $\mathcal{F}=\left\{f_n: n\in\mathbb{N}\right\}$ be a family of holomorphic functions in $\mathbb{D}$ given by $$f_n(z)=\cos\left(e^{z+n}\right).$$ Let $S_1=S_2=\left\{-1,1\right\}$ and $P[f_n]=f_n'.$ Then for each $a\in S_1,$ $f_n-a$ has zeros of multiplicity $2$ and $P[f_n](z)\in S_1\not\Rightarrow f_n(z)\in S_2.$ However, the family $\mathcal{F}$ is not normal in $\mathbb{D}$ since along real axis, the family $\mathcal{F}$ is uniformly bounded by $1$ but along imaginary axis, the limit function is $\infty.$
\end{example}

The condition ``for each $a\in S_1,$ $f-a$ has zeros of multiplicity at least $\alpha+1$" in Theorem \ref{thm:2} is not redundant as demonstrated by the following example (one may also see \cite[p.193]{hayman-storvick}):

\begin{example}
Consider $$f(z)=2(1-z)~\mbox{exp}\left\{\frac{2+z}{1-z}\right\}$$ in the open unit disk $\mathbb{D}.$ Then $$f'(z)=\frac{4+2z}{1-z}~\mbox{exp}\left\{\frac{2+z}{1-z}\right\}$$ and $f(z)\neq 0$ in $\mathbb{D}.$ Also, one can easily see that $\lvert f'(z)\rvert>\sqrt{e}$ in $\mathbb{D}.$ Let $S_1=\left\{0,a_1,a_2\right\},$ where $a_i$'s $(i=1,2)$ are non-zero distinct complex numbers with $\lvert a_i\rvert<\sqrt{e}$ and $S_2$ be any finite set in $\mathbb{C}.$ Note that $f-a_i$ does not have zeros of multiplicity at least $2.$ Let $P[f]=f'$ and $a_3$ be any non-zero complex number distinct from $a_1$ and $a_2$ such that $\lvert a_3\rvert<\sqrt{e}.$ Clearly, $\left\lvert f''(z)\right\rvert\leq M$ whenever $f'(z)=a_3$ for any $M>0$ and $P\left[f\right](z)\in S_1\Rightarrow f(z)\in S_2.$ However, $f$ is not a normal function since if $$z=\frac{1}{2}(1+e^{i\theta}),~0<\theta<2\pi,$$ $$\mbox{then}~~~\lvert f(z)\rvert=e^2\lvert\sin\left(\theta/2\right)\rvert\longrightarrow 0,~\mbox{as}~\theta\longrightarrow 0,$$ whereas $$f(z)\longrightarrow\infty,~\mbox{as}~z\longrightarrow 1^-~\mbox{through real values}.$$ 
\end{example}

The following two results establish that the cardinality of the set $S_1$ in Theorem \ref{thm:1} and Theorem \ref{thm:2} can be reduced by one under suitable conditions.

\begin{theorem}\label{thm:3}
Let $\mathcal{F}\subset\mathcal{M}(D),$ $S_1$ and $S_2$ be two finite subsets of $\mathbb{C}$ with $S_1=\left\{a_1, a_2\right\},$ where $a_1,~a_2$ are non-zero distinct complex numbers. Let $P$ be a differential polynomial defined in \eqref{eqn:1} and satisfying \eqref{eqn:2}. Suppose that for each $f\in\mathcal{F},$ $f-a_i~(i=1,2)$ has zeros of multiplicity at least $\alpha+1.$ If $$P\left[f\right](z)\in S_1\Rightarrow f(z)\in S_2$$ in $D$ and if there exist some $M>0$ and a point $a_3\in\mathbb{C}\setminus S_1$ such that $\lvert f'(z)\rvert\leq M$ whenever $f(z)=a_3,$ then $\mathcal{F}$ is normal in $D.$
\end{theorem}

By Remark $1,$ it follows that $S_1$ may contain zero if the differential polynomial $P$ in Theorem \ref{thm:3} is replaced by $\left(f^{(k)}\right)^m$ and $\alpha$ replaced by $k,$ where $k,~m$ are positive integers. However, we cannot take $a_1=a_2$ as the following example demonstrates:

\begin{example}
Let $a$ be any non-zero complex number and let $\mathcal{F}=\left\{f_n: n\in\mathbb{N}\right\}$ be a family of meromorphic functions in $\mathbb{D}$ given by $$f_n(z)=\frac{a}{a+e^{nz}}$$ and let $S_1=\left\{0\right\}.$ Then $P[f_n](z)=f_n'(z)\in S_1\Rightarrow f_n(z)\in S_2$ for any finite set $S_2\subset\mathbb{C}$ and $f_n$ has zeros of arbitrary multiplicity. Also, we have $\lvert f_n'(z)\rvert\leq M$ whenever $f_n(z)=1$ for any positive constant $M.$ However, the family $\mathcal{F}$ is not normal in $\mathbb{D}.$
\end{example}

%A criterion for normal meromorphic functions corresponding to Theorem \ref{thm:3} is obtained as

\begin{theorem}\label{thm:4}
Let $f\in\mathcal{M}(\mathbb{D}),$ $S_1$ and $S_2$ be two finite subsets of $\mathbb{C}$ with $S_1=\left\{a_1, a_2\right\},$ where $a_1,~a_2$ are non-zero distinct complex numbers. Let $P$ be a differential polynomial defined in (\ref{eqn:1}) and satisfying \eqref{eqn:2}. Suppose that $f-a_i~(i=1,2)$ has zeros of multiplicity at least $\alpha+1.$ If $$P\left[f\right](z)\in S_1\Rightarrow f(z)\in S_2$$ in $\mathbb{D}$ and if there exist some $M>0$ and a point $a_3\in\mathbb{C}\setminus S_1$ such that $\lvert f'(z)\rvert\leq M$ whenever $f(z)=a_3,$ then $f$ is a normal function.
\end{theorem}

Finally, we have another criterion for normal meromorphic functions and is of independent interest.

\begin{theorem}\label{thm:5}
Let $f\in\mathcal{M}(\mathbb{D}),$ $S_1=\left\{a_1, a_2, a_3\right\}$ and $S_2=\left\{b_1,b_2,b_3\right\}$ be two finite subsets of $\mathbb{C}.$ Let $l, m\in\mathbb{N}$ and $a$ be any fixed complex number. Then $f$ is a normal function if any one of the following holds:
\begin{itemize}
\item[$(i)$] $\left(f^{(l)}\right)^m(z)\in S_1\Leftrightarrow f(z)\in S_2,$ and $f-a$ has zeros and poles of multiplicity at least $l+1,$ if $a\in S_1.$

\item[$(ii)$] $f(z)\in S_1\Rightarrow \left(f^{(l)}\right)^m(z)\in S_2,$ and $f-a$ has zeros and poles of multiplicity at least $l+2,$ if $a\notin S_1.$
\end{itemize}
\end{theorem}

\section{Auxiliary Results}
 In this section, we describe some preliminary results that are crucial to prove main results of this paper. First recall that if $f\in\mathcal{M}(\mathbb{C})$ and $a\in\mathbb{C}_{\infty},$ then $a$ is said to be a totally ramified value of $f$ if $f-a$ has no simple zeros. Nevanlinna (see \cite[p. 84]{bergweiler}) proved the following widely known result concerning multiplicities of $a$-points of a meromorphic function. This result plays a major role in the proofs of Theorem \ref{thm:3} and Theorem \ref{thm:4}:
 
\begin{lemma}[Nevanlinna's Theorem]\label{neva} Let $f$ be a non constant meromorphic function, $a_1, a_2, \ldots, a_q\in\mathbb{C}_{\infty}$ and $m_1, m_2,\ldots, m_q\in\mathbb{N}.$ Suppose that all $a_j$-points of $f$ have multiplicity at least $m_j,$ for $j=1,2,\ldots, q.$ Then $$\sum\limits_{j=1}^{q}\left(1-\frac{1}{m_j}\right)\leq 2.$$ If $f$ omits the value $a_j$, then $m_j=\infty.$
\end{lemma}

We need the following rescaling lemma due to Lohwater and Pommerenke \cite[Theorem 1]{lohwater}:

\begin{lemma}\label{lem:1}
Let $f\in\mathcal{M}(\mathbb{D}).$ Suppose that $f$ is not a normal function. Then there exist points $z_n\in\mathbb{D}$ and positive numbers $\rho_n\longrightarrow 0$ such that $$g_n(\zeta)=f(z_n+\rho_n\zeta)\longrightarrow g(\zeta)$$ locally uniformly in $\mathbb{C}$ with respect to the spherical metric, where $g$ is a non-constant meromorphic function on $\mathbb{C}.$
\end{lemma}

The proofs of main results in this paper rely essentially on the following extension of the famous Zalcman-Pang Lemma due to Chen and Gu \cite{chen-gu} (see also \cite[p. 216]{zalcman-2}, cf. \cite[Lemma 2]{pang-zalcman}).

\begin{lemma}[Zalcman-Pang Lemma]\label{lem:2}
Let $\mathcal{F}\subset\mathcal{M}(D)$ be a family of meromorphic functions all of whose zeros have multiplicities at least $m$ and whose poles have multiplicities at least $p.$ Let $-p<\alpha<m.$ If $\mathcal{F}$ is not normal at $z_0\in D,$ then there exist
sequences $\left\{f_n\right\}\subset\mathcal{F},$ $\left\{z_n\right\}\subset D$ satisfying $z_n\longrightarrow z_0$ and positive numbers $\rho_n$ with $\rho_n\longrightarrow 0$ such that the sequence $\left\{g_n\right\}$ defined by $$g_n(\zeta)=\rho_{n}^{-\alpha}f_n(z_n+\rho_n\zeta)\longrightarrow g(\zeta)$$ locally uniformly in $\mathbb{C}$ with respect to the spherical metric, where $g$ is a non-constant meromorphic function on $\mathbb{C}$ such that for every $\zeta\in\mathbb{C},$ $g^{\#}(\zeta)\leq g^{\#}(0)=1.$
 \end{lemma}

The following lemma due to Doeringer \cite[Lemma 1 (i)]{doeringer} is an extension of the Nevanlinna's well known Lemma on the Logarithmic Derivative \cite[Lemma 2.3]{hayman}.

\begin{lemma}\label{lem:3}
Let $f\in\mathcal{M}(\mathbb{C})$ and $P[f]$ be a differential polynomial of $f$ with meromorphic coefficients $a_j,~j=1,\ldots, k.$ Then $$m\left(r, P[f]\right)\leq d(Q)\cdot m(r,f) + \sum\limits_{j=1}^{k} m(r, a_j) + S(r,f).$$
\end{lemma}

We also need the following lemma in the proofs of Theorem \ref{thm:1} and Theorem \ref{thm:2}.

\begin{lemma}\label{lem:4}
Let $f\in\mathcal{M}(\mathbb{C})$ be such that $f$ has finitely many zeros and let $P[f]$ be a differential polynomial of $f$ with coefficients $a_j,~j=1,\ldots, k$ such that $T(r, a_j)=S(r, f).$  If $P[f]$ is constant, then either $f$ is a polynomial or $P[f]\equiv 0.$
\end{lemma}

\begin{proof}
Since $f$ has finitely many zeros, write $f=g/h,$ for some polynomial $g$ and some entire function $h.$ Then we have 
$$f^{(k)}=\frac{Q_k[h]}{h^{k+1}},$$ where $Q_k[h]$ is a homogeneous differential polynomial of degree $k$ with polynomial coefficients and $$M_j[f]=\frac{H_j[h]}{h^{w(M_j)}}$$ for $j=1,2,\ldots, k,$ where $H_j[h]$ are homogeneous differential polynomials such that $d(H_j)=w(M_j)-d(M_j).$ This gives $$P[f]=\frac{1}{h^{w(P)}}\cdot\left(\sum\limits_{j=1}^{k}a_j h^{w(P)-w(M_j)}H_j[h]\right)=\frac{U[h]}{h^{w(P)}},$$ where $U[h]$ is a differential polynomial of $h$ whose coefficients are the product of polynomials with small functions of $f.$ Moreover $$d(U)\leq\max\limits_{1\leq j\leq k}\left\{w(P)-d(M_j)\right\}\leq w(P)-1.$$ Now suppose that $P[f]\equiv c$ for some non-zero constant $c.$ Then we can write 
\begin{equation}\label{eqn:3}
h^{w(P)}=\frac{1}{c}\cdot U[h].
\end{equation}

From Lemma \ref{lem:3}, we deduce that $$w(P)\cdot m(r,h)=m(r, U[h])+ \mathcal{O}(1)\leq (w(P)-1)\cdot m(r,h) + \mathcal{O}(\log r) + S(r, f),$$ showing that $T(r,h)=m(r,h)=S(r,f)=S(r,h).$ Therefore, $h$ must be a polynomial and so is $h^{w(P)}.$ But degree of $h^{w(P)}$ is $w(P)\cdot d(h)$ and $U[h]$ is a polynomial of degree at most $\left(w(P)-1\right)\cdot d(h).$ This together with \eqref{eqn:3} yield $d(h)\leq 0.$ Thus $h$ is constant and so $f$ is a polynomial.
\end{proof}

The proofs of Theorem \ref{thm:1} and Theorem \ref{thm:2} also require the following result due to Grahl \cite[Lemma 13]{grahl}.
 
\begin{lemma}\label{lem:5}
Let $f\in\mathcal{H}(\mathbb{C})$ be a non-constant function and let $P$ be a differential polynomial such that $P[f]\equiv 0.$ Assume that $T(r,a)=S(r,f)$ for each coefficient $a$ of $P[f].$ If $$P= Q_0 + \ldots + Q_d$$ with homogeneous differential polynomials $Q_j$ of degree $j$ or ($Q_j\equiv 0$) and $Q_l[f]\not\equiv 0$ for some $l$ in the set $\left\{0,\ldots, d-1\right\},$ then $$m\left(r,\frac{1}{f}\right)\leq l\cdot N\left(r,\frac{1}{f}\right) + S(r,f).$$
\end{lemma}

The following value distribution result plays an important role in the proof of Theorem \ref{thm:5}:

\begin{lemma}\label{lem:7}
Let $f\in\mathcal{M}(\mathbb{C}),$ $f\not\equiv$ constant, $l, q\in\mathbb{N}$ and let $S=\left\{a_1, a_2,\ldots,a_q\right\}$ be any finite set in $\mathbb{C}$ such that $$f(z)\in S\Rightarrow f^{(l)}(z)=0.$$ Then $$q~T(r, f)< (l+1)\overline{N}(r,f)+ \overline{N}\left(r,\frac{1}{f}\right)+ N\left(r,\frac{1}{f}\right) + S(r,f).$$ 
\end{lemma}

We shall prove Lemma \ref{lem:7} by using the following value distribution result due to Yi \cite[Lemma 3]{yi}:

\begin{lemma}\label{lem:6}
Let $f$ be a non-constant meromorphic function in $\mathbb{C}$ and let $l$ be a non-negative integer. Then $$N\left(r,\frac{1}{f^{(l)}}\right)< N\left(r,\frac{1}{f}\right)+l\bar{N}\left(r, f\right)+S\left(r, f\right).$$
	
\end{lemma}

\begin{proof}[\bf Proof of Lemma \ref{lem:7}]
By the Second Fundamental Theorem of Nevanlinna, we have $$q~T(r,f)\leq\overline{N}(r,f)+\overline{N}\left(r, \frac{1}{f}\right)+\sum\limits_{i=1}^{q}\overline{N}\left(r,\frac{1}{f-a_i}\right) + S(r,f).$$
Since $$f(z)\in S\Rightarrow f^{(l)}(z)=0,$$ we find that $$\sum\limits_{i=1}^{q}\overline{N}\left(r,\frac{1}{f-a_i}\right)\leq\overline{N}\left(r,\frac{1}{f^{(l)}}\right)$$ and hence we obtain
\begin{eqnarray*}
q~T(r,f)&\leq&\overline{N}(r,f)+\overline{N}\left(r,\frac{1}{f}\right)+\sum\limits_{i=1}^{q}\overline{N}\left(r,\frac{1}{f-a_i}\right)+S(r,f)\\
&\leq& \overline{N}(r,f)+\overline{N}\left(r,\frac{1}{f}\right)+ \overline{N}\left(r,\frac{1}{f^{(l)}}\right)+S(r,f)\\
&\leq& \overline{N}(r,f)+ \overline{N}\left(r,\frac{1}{f}\right)+ N\left(r,\frac{1}{f^{(l)}}\right) +S(r,f)
\end{eqnarray*}
By Lemmaa \ref{lem:6}, we have $$N\left(r,\frac{1}{f^{(l)}}\right)< N\left(r,\frac{1}{f}\right)+l\bar{N}\left(r, f\right)+S\left(r, f\right)$$ and therefore
\begin{eqnarray*}
q~T(r,f)&<& \overline{N}(r,f)+\overline{N}\left(r,\frac{1}{f}\right)+ N\left(r,\frac{1}{f}\right)+l\bar{N}\left(r, f\right)+ S(r,f)\\
&\leq& (l+1)\overline{N}(r,f)+ \overline{N}\left(r,\frac{1}{f}\right)+ N\left(r,\frac{1}{f}\right) + S(r,f).
\end{eqnarray*}
\end{proof}

\section{Proofs of Main Results}

\begin{proof}[\bf Proof of Theorem \ref{thm:1}]
Suppose that $\mathcal{F}$ is not normal at $z_0\in D.$ Then by Lemma \ref{lem:2}, there exist sequences $\left\{f_n\right\}\subset\mathcal{F},$ $\left\{z_n\right\}\subset D$ with $z_n\longrightarrow z_0$ and positive numbers $\rho_n$ satisfying $\rho_n\longrightarrow 0$ such that $$g_n(\zeta):= f_n(z_n+\rho_n\zeta)\longrightarrow g(\zeta)$$ locally uniformly in $\mathbb{C}$ with respect to the spherical metric, where $g\in\mathcal{M}(\mathbb{C})$ is non-constant. By Picard's Theorem, it follows that $g$ assumes at least one of the values of $S_1.$ We claim that for any $a\in S_1,$ all zeros of $g-a$ have multiplicity at least $\alpha+1.$ Indeed, let $g(\zeta_0)=a.$ Since $g\not\equiv a,$ by Hurwitz's Theorem, there exists $\zeta_n\longrightarrow \zeta_0$ such that for sufficiently large $n,$ 
$$g_n(\zeta_n)=f_n(z_n+\rho_n\zeta_n)=a\in S_1.$$ By hypothesis, we have $f_n^{(i)}(z_n+\rho_n\zeta_n)=0$ for $1\leq i\leq\alpha.$ 

Thus $$g^{(i)}(\zeta_0)=\lim\limits_{n\to\infty}g_n^{(i)}(\zeta_n)=\lim\limits_{n\to\infty} \rho_n^{i}f_n^{(i)}(z_n+\rho_n\zeta_n)=0$$ for $1\leq i\leq\alpha,$ and this establishes the claim.

Now suppose that $\zeta_0$ is a zero of $g-a$ with multiplicity $l.$ Then there exist some $\delta>0$ such that for sufficiently large $n,$ $g_n$ is holomorphic in the disk $D(\zeta_0, \delta).$ Let $$ h_n(\zeta):=\frac{g_n(\zeta)-a}{\rho_n^\alpha}.$$ Then $h_n$ is holomorphic in $D(\zeta_0, \delta)$ and $h_n(\zeta')=0$ if and only if $g_n(\zeta')=a$ and so $h_n$ has zeros of multiplicity at least $\alpha+1.$

 Next, we claim that $\left\{h_n\right\}$ is not normal at $\zeta_0.$ Suppose otherwise. Then there exist a $\delta_1$ such that $0<\delta_1<\delta$ and a subsequence of $\left\{h_n\right\}$ (again denoted by $\left\{h_n\right\}$) such that $h_n\longrightarrow h$ locally uniformly in $D(\zeta_0, \delta_1),$ where $h$ is either holomorphic or identically $\infty$ in $D(\zeta_0, \delta_1).$ Since $g(\zeta_0)=a$ and $g\not\equiv a,$ by Hurwitz's Theorem, we find that
$h(\zeta_0)=0.$ Also, since zeros of $g-a$ in $D(\zeta_0, \delta_1)$ are isolated, there is some $\zeta_1\neq\zeta_0$ in $D(\zeta_0,\delta_1)$ such that $g(\zeta_1)\neq a.$ Thus for sufficiently large $n,$ $\lvert g_n(\zeta_1)-a\rvert>0$ and hence $$\lvert h_n(\zeta_1)\rvert=\frac{\lvert g_n(\zeta_1)-a\rvert}{\rho_n^{\alpha}}\longrightarrow\infty.$$ This implies that $\left\{h_n\right\}$ converges uniformly to $\infty$ in $D(\zeta_0,\delta_1)$ which is not the case. 

Again, by Lemma \ref{lem:2}, there exists a subsequence of $\left\{h_n\right\}$ (again denoted by $\left\{h_n\right\}$), a sequence of points $\hat{\zeta_n}\longrightarrow\zeta_0$ and positive numbers $r_n\longrightarrow 0$ such that $$\phi_n(\xi)=\frac{h_n(\hat{\zeta_n}+r_n\xi)}{r_n^\alpha}=\frac{g_n(\hat{\zeta_n}+r_n\xi)-a}{(\rho_n r_n)^\alpha}=\frac{f_n(z_n+\rho_n\hat{\zeta_n}+\rho_n r_n\xi)-a}{(\rho_n r_n)^\alpha}$$ converges locally uniformly to a non-constant entire function $\phi(\xi).$ Since $g-a$ has zeros of multiplicity at least $\alpha+1,$ it easily follows that zeros of $\phi$ have multiplicity at least $\alpha+1.$

\smallskip

{\it Claim 1}: $\phi$ has finitely many zeros.\\
It is sufficient to show that $\phi$ has at most $l$ distinct zeros. Suppose on the contrary that $\phi$ has $l+1$ distinct zeros, say $\xi_1,\xi_2,\ldots, \xi_{l+1}.$ Then by Hurwitz's Theorem, there exist $l+1$ distinct sequences $\left\{\xi_{n_j}\right\}$ such that $\xi_{n_j}\longrightarrow\xi_j$ and $\phi_n\left(\xi_{n_j}\right)=0$ for $j=1,2,\ldots, l+1.$ This implies that $$g_n(\hat{\zeta_n}+r_n\xi_{n_j})=a,~j=1,2,\ldots,l+1.$$ Since $\hat{\zeta_n}+r_n\xi_{n_j}\longrightarrow\zeta_0$ and $\hat{\zeta_n}+r_n\xi_{n_i}\neq \hat{\zeta_n}+r_n\xi_{n_j}$ for $i\neq j,$ it follows that $\zeta_0$ is a zero of $g-a$ with multiplicity at least $l+1$ which contradicts the fact that $\zeta_0$ is a zero of $g-a$ with multiplicity $l.$ This proves the claim.

\smallskip

{\it Claim 2}: If $P[\phi]$ assumes a non-zero finite value, say $\beta,$ then $S_\beta=\left\{f_n(z): P[f_n](z)=\beta, z\in D\right\}$ is an infinite set. \\
First note that 
\begin{eqnarray*}
\tilde{P}[\phi_n](\xi) &:=& P[f_n](z_n+\rho_n\hat{\zeta_n}+\rho_n r_n\xi)\\
&=& \sum\limits_{i=1}^{m}a_i(z_n+\rho_n\hat{\zeta_n}+\rho_n r_n\xi)~\rho_n^{\left[(1+\alpha)d(M_i)-w(M_i)\right]}~M_i[\phi_n](\xi).
\end{eqnarray*}
Since $$1+\alpha=\frac{w(M_1)}{d(M_1)}~\mbox{ and }~ \frac{w(M_1)}{d(M_1)}\geq\frac{w(M_t)}{d(M_t)}~ \mbox{ for }~ 2\leq t\leq m,$$ we assume, without loss of generality, that $$1+\alpha=\frac{w(M_1)}{d(M_1)}=\frac{w(M_2)}{d(M_2)}=\cdots=\frac{w(M_s)}{d(M_s)}$$ and $$\frac{w(M_1)}{d(M_1)}>\frac{w(M_t)}{d(M_t)}~ \mbox{for}~ s+1\leq t\leq m.$$
Therefore, we obtain
\begin{eqnarray*}
\tilde{P}[\phi_n](\xi) &=& P[f_n](z_n+\rho_n\hat{\zeta_n}+\rho_n r_n\xi)\\
&=& \sum\limits_{i=1}^{s}a_i(z_n+\rho_n\hat{\zeta_n}+\rho_n r_n\xi)M_i[\phi_n](\xi)\\ &\qquad& + \sum\limits_{i=s+1}^{m}a_i(z_n+\rho_n\hat{\zeta_n}+\rho_n r_n\xi)~\rho_n^{\left[(1+\alpha)d(M_i)-w(M_i)\right]}~M_i[\phi_n](\xi)\\
&=& \tilde{Q}[\phi_n](\xi) + \sum\limits_{i=s+1}^{m}a_i(z_n+\rho_n\hat{\zeta_n}+\rho_n r_n\xi)~\rho_n^{\left[(1+\alpha)d(M_i)-w(M_i)\right]}M_i[\phi_n](\xi),
\end{eqnarray*}
where $$\tilde{Q}[\phi_n](\xi):= \sum\limits_{i=1}^{s}a_i(z_n+\rho_n\hat{\zeta_n}+\rho_n r_n\xi)M_i[\phi_n](\xi).$$
Again, since all $a_i~(1\leq i\leq m)$ are holomorphic functions in $D$ and $a_i(z)\neq 0$ for $1\leq i\leq s,$ it follows that $$\sum\limits_{i=s+1}^{m}a_i(z_n+\rho_n\hat{\zeta_n}+\rho_n r_n\xi)~\rho_n^{\left[(1+\alpha)d(M_i)-w(M_i)\right]}~M_i[\phi_n](\xi)$$ converges uniformly to $0$ on compact subsets of $\mathbb{C}$ and hence $$\tilde{P}[\phi_n](\xi)=P[f_n](z_n+\rho_n\hat{\zeta_n}+\rho_n r_n\xi)\longrightarrow Q[\phi](\xi)$$ uniformly on compact subsets of $\mathbb{C},$ where $$Q[\phi](\xi)=\sum\limits_{i=1}^{s}a_i(z_0)M_i[\phi](\xi).$$

We claim that $Q[\phi]\not\equiv$ constant. If $Q[\phi]\equiv$ constant, then by Lemma \ref{lem:4}, we find that either $\phi$ is polynomial or $Q[\phi]\equiv 0.$ If $\phi$ is a polynomial, then by Weierstrass's Theorem, we have $$\phi_n^{(\alpha)}(\xi)=f_n^{(\alpha)}(z_n+\rho_n\hat{\zeta_n}+\rho_n r_n\xi)\longrightarrow\phi^{(\alpha)}(\xi)$$ uniformly on compact subsets of $\mathbb{C}.$ Since $\phi$ is a non-constant polynomial having zeros of multiplicity at least $\alpha+1,$ it follows that $\phi^{(\alpha)}$ assumes every value in $\mathbb{C}.$ In particular, $\phi^{(\alpha)}$ assumes a non-zero value $b\in S_1.$ Let $\phi^{(\alpha)}(\xi^*)=b.$ Since $\phi^{(\alpha)}(\xi)\not\equiv b,$ by Hurwitz's Theorem, there exists $\xi_n^*\longrightarrow\xi^*$ such that for sufficiently large $n,$ $$\phi_n^{(\alpha)}(\xi_n^*)=f_n^{(\alpha)}(z_n+\rho_n\hat{\zeta_n}+\rho_n r_n\xi_n^*)=b~(\neq 0).$$ This implies that $f_n(z_n+\rho_n\hat{\zeta_n}+\rho_n r_n\xi_n^*)\neq a,$ since if $f_n(z_n+\rho_n\hat{\zeta_n}+\rho_n r_n\xi_n^*)=a,$ then by hypothesis, we have $f_n^{(\alpha)}(z_n+\rho_n\hat{\zeta_n}+\rho_n r_n\xi_n^*)=0,$ a contradiction to the fact that $f_n^{(\alpha)}(z_n+\rho_n\hat{\zeta_n}+\rho_n r_n\xi_n^*)=b\neq 0.$

Then 
$$\phi(\xi^*)=\lim\limits_{n\to\infty}\phi_n(\xi_n^*)=\lim\limits_{n\to\infty}\frac{f_n(z_n+\rho_n\hat{\zeta_n}+\rho_n r_n\xi_n^*)-a}{(\rho_n r_n)^\alpha}=\infty,$$
showing that $\xi^*$ is a pole of $\phi$ which is not possible. Thus $\phi$ cannot be a polynomial. 

Next, if $Q[\phi]\equiv 0,$ then by Lemma \ref{lem:5}, we deduce that $$m(r,\frac{1}{\phi})=S(r,\phi)$$ and hence $$m(r,\phi)=S(r,\phi)$$ showing that $\phi$ is a polynomial, which is not the case. Thus $Q[\phi]\not\equiv 0$ and this proves our claim. Then by Picard's Theorem, $Q[\phi](\xi)=\beta$ has at least one solution for any $\beta\in\mathbb{C}$ with at most one exception.

Now, let $\xi_0\in\mathbb{C}$ be such that $Q[\phi](\xi_0)=\beta_1~(\neq 0)$ and $\phi(\xi_0)=\beta_2,$ where $\beta_1,\beta_2\in\mathbb{C}.$ It is clear that $\beta_2\neq 0,$ otherwise $Q[\phi](\xi_0)=0$ since zeros of $\phi$ have multiplicity at least $\alpha+1.$ 

Since $Q[\phi]\not\equiv\beta_1,$ by Hurwitz's Theorem, there exists $\xi_n\longrightarrow\xi_0$ such that for sufficiently large $n,$ $$P[f_n](z_n+\rho_n\hat{\zeta_n}+\rho_n r_n\xi_n)=\beta_1~(\neq 0).$$ Also, $\phi_n(\xi_n)\longrightarrow\phi(\xi_0)=\beta_2~(\neq 0)$ implies that $f_n(z_n+\rho_n\hat{\zeta_n}+\rho_n r_n\xi_n)\longrightarrow a.$ Further, note that if $f_n(z_n+\rho_n\hat{\zeta_n}+\rho_n r_n\xi_n)= a,$ then by hypothesis, $$f_n^{(i)}(z_n+\rho_n\hat{\zeta_n}+\rho_n r_n\xi_n)=0~ \mbox{for}~ 1\leq i\leq\alpha$$ so that $$P[f_n](z_n+\rho_n\hat{\zeta_n}+\rho_n r_n\xi_n)=0,$$ a contradiction to the fact that $P[f_n](z_n+\rho_n\hat{\zeta_n}+\rho_n r_n\xi_n)=\beta_1\neq 0.$ Therefore, $f_n(z_n+\rho_n\hat{\zeta_n}+\rho_n r_n\xi_n)\neq a$ and hence $$S_{\beta_{1}}=\left\{f_n(z_n+\rho_n\hat{\zeta_n}+\rho_n r_n\xi_n): P[f_n](z_n+\rho_n\hat{\zeta_n}+\rho_n r_n\xi_n)=\beta_1\right\}$$ is an infinite set. This proves Claim $2.$

Now, since $\#(S_1)\geq 3,$ $S_1$ contains a non-zero finite value, say $\nu.$ Then by Claim $2,$ the set $S_\nu$ is an infinite set and cannot be contained in the finite set $S_2,$ a contradiction to our hypothesis. Hence $\mathcal{F}$ is normal in $D.$ 
\end{proof}

\begin{proof}[\bf Proof of Theorem \ref{thm:2}]
Suppose that $f$ is not a normal function. Then by Lemma \ref{lem:1}, there exist points $z_n\in\mathbb{D}$ and positive numbers $\rho_n\longrightarrow 0$ such that $$g_n(\zeta):=f(z_n+\rho_n\zeta)$$ converges spherically locally uniformly in $\mathbb{C}$ to a non-constant meromorphic function $g(\zeta).$ Rest of the proof goes on the lines of the proof of Theorem \ref{thm:1} with simple modifications, hence omitted.
\end{proof}

\begin{proof}[\bf Proof of Theorem \ref{thm:3}]
Suppose that $\mathcal{F}$ is not normal at $z_0\in D.$ Then by Lemma \ref{lem:2}, there exist sequences $\left\{f_n\right\}\subset\mathcal{F},$ $\left\{z_n\right\}\subset D$ with $z_n\longrightarrow z_0$ and positive numbers $\rho_n$ satisfying $\rho_n\longrightarrow 0$ such that $$g_n(\zeta):= f_n(z_n+\rho_n\zeta)\longrightarrow g(\zeta)$$ locally uniformly in $\mathbb{C}$ with respect to the spherical metric, where $g\in\mathcal{M}(\mathbb{C})$ is non-constant.

We claim that $g$ assumes at least one value from $S_1.$ Suppose on the contrary that $g$ omits both $a_1$ and $a_2.$ Then by Picard's Theorem, $g$ assumes every value in $\mathbb{C}\setminus S_1.$ In particular, $g$ must assume the value $a_3.$ Suppose there exists $\zeta_0\in\mathbb{C}$ such that $g(\zeta_0)=a_3.$ Since $g\not\equiv a_3,$ by Hurwitz's Theorem, there exists a sequence $\zeta_n\longrightarrow\zeta_0$ such that for sufficiently large $n,$ $$g_n(\zeta_n)=f_n(z_n+\rho_n\zeta_n)=a_3.$$ By hypothesis, we have $\lvert f_n'(z_n+\rho_n\zeta_n)\rvert\leq M.$ 

Then
\begin{eqnarray*}
\left\lvert g'(\zeta_0)\right\rvert&=&\left\lvert\lim\limits_{n\to\infty}g_n'(\zeta_n)\right\rvert\\
&=& \lim\limits_{n\to\infty}\left\lvert \rho_n f_n'(z_n+\rho_n\zeta_n)\right\rvert\\
&\leq& \lim\limits_{n\to\infty}\rho_n~ M\\
&=& 0.
\end{eqnarray*}
This shows that $\zeta_0$ is a zero of $g-a_3$ of multiplicity at least $2.$ Let $m_1,~ m_2$ and $m_3$ denote the multiplicity of zeros of $g-a_1,~ g-a_2$ and $g-a_3$ respectively. Then by simple calculation, we find that $$\sum\limits_{j=1}^{3}\left(1-\frac{1}{m_j}\right)>2,$$ which is not true by Lemma \ref{neva}. Now the proof is completed by following the proof of Theorem \ref{thm:1}.
\end{proof}

\medskip

 The proof of Theorem \ref{thm:4} follows from the proof of Theorem \ref{thm:3} with minor modifications and hence omitted.

\medskip

\begin{proof}[\bf Proof of Theorem \ref{thm:5}]
$(i)$ Suppose that $f$ is not a normal function. Then the function $g=f-a$ is not normal. So, by Lemma \ref{lem:1}, there exist points $z_n\in\mathbb{D}$ and numbers $\rho_n\longrightarrow 0$ such that $$h_n(\zeta):=g(z_n+\rho_n\zeta)\longrightarrow h(\zeta)$$ spherically locally uniformly in $\mathbb{C},$ where $h\in\mathcal{M}(\mathbb{C})$ is a non-constant function. 

By Argument principle, it follows that both zeros and poles of $h$ have multiplicity at least $l+1.$ Since $a\in S_1,$ we assume, without loss of generality, that $a=a_1.$ Then $$h_n(\zeta)=g(z_n+\rho_n\zeta)=f(z_n+\rho_n\zeta)-a_1\longrightarrow h(\zeta).$$ 

\textit{Case 1.} $h$ does not omit zero.

Let $h(\zeta_0)=0.$ Then we can find some $\delta>0$ such that for sufficiently large $n,$ $h_n$ is holomorphic in $D(\zeta_0,\delta).$

Let $$\psi_n(\zeta)=\frac{h_n(\zeta)}{\rho_n^l}.$$ Then $\psi_n$ is holomorphic in $D(\zeta_0,\delta).$ 
Also, $\psi_n(\zeta')=0$ if and only if $h_n(\zeta')=0$ and so $h_n$ has zeros of multiplicity at least $l+1.$

 Next, we claim that $\left\{\psi_n\right\}$ is not normal at $\zeta_0.$ Suppose on the contrary that $\left\{\psi_n\right\}$ is normal. Then there exist a $\delta_1$ such that $0<\delta_1<\delta$ and a subsequence of $\left\{\psi_n\right\}$ (again denoted by $\left\{\psi_n\right\}$) such that $\psi_n\longrightarrow \psi$ locally uniformly in $D(\zeta_0, \delta_1),$ where $\psi$ is either holomorphic or identically $\infty$ in $D(\zeta_0, \delta_1).$ 

Since $h(\zeta_0)=0$ and $h\not\equiv 0,$ by Hurwitz's Theorem, there exits $\zeta_n\longrightarrow\zeta_0$ such that for sufficiently large $n,$ $h_n(\zeta_n)=0$ and hence 
$$\psi(\zeta_0)=\lim\limits_{n\to\infty}\psi_n(\zeta_n)=\lim\limits_{n\to\infty}\frac{h_n(\zeta_n)}{\rho_n^l}=0.$$
Also, since zeros of $h$ in $D(\zeta_0, \delta_1)$ are isolated, there is some $\zeta_1\neq\zeta_0$ in $D(\zeta_0,\delta_1)$ such that $h(\zeta_1)\neq 0.$ Thus for sufficiently large $n,$ $\lvert h_n(\zeta_1)\rvert>0$ and hence $$\lvert \psi_n(\zeta_1)\rvert=\frac{\lvert h_n(\zeta_1)\rvert}{\rho_n^l}\longrightarrow\infty.$$ This implies that $\left\{\psi_n\right\}$ converges uniformly to $\infty$ in $D(\zeta_0,\delta_1)$ which contradicts the fact that $\psi(\zeta_0)=0.$ 

Now, by Lemma \ref{lem:2}, there exists a subsequence of $\left\{\psi_n\right\}$ (again denoted by $\left\{\psi_n\right\}$), a sequence of  points $\hat{\zeta_n}\longrightarrow\zeta_0$ and positive numbers $r_n\longrightarrow 0$ such that $$\phi_n(\xi)=\frac{\psi_n(\hat{\zeta_n}+r_n\xi)}{r_n^l}=\frac{h_n(\hat{\zeta_n}+r_n\xi)}{(\rho_n r_n)^l}=\frac{f_n(z_n+\rho_n\hat{\zeta_n}+\rho_n r_n\xi)-a_1}{(\rho_n r_n)^l}$$ converges locally uniformly to a non-constant entire function $\phi(\xi).$ It is easy to see that $\phi$ has zeros of multiplicity at least $l+1.$ By Weierstrass's Theorem, it follows that $$\left(\phi_n^{(l)}\right)^m(\xi)=\left(f^{(l)}\right)^m(z_n+\rho_n\hat{\zeta_n}+\rho_n r_n\xi)\longrightarrow\left(\phi^{(l)}\right)^m(\xi)$$ uniformly on compact subsets of $\mathbb{C}.$

Since $\#(S_1)=3,$ by Picard's Theorem, it follows that $\left(\phi^{(l)}\right)^m$ can omit at most one value from $S_1.$ Pick a non-zero value from $S_1,$ say $a_2$ such that $\left(\phi^{(l)}\right)^m(\xi)=a_2$ has a solution.

Let $\xi_0\in\mathbb{C}$ be such that $\left(\phi^{(l)}\right)^m(\xi_0)=a_2~(\neq 0)$ and $\phi(\xi_0)=a_2^*$ where $a_2,a_2^*\in\mathbb{C}.$ Clearly, $a_2^*\neq 0,$ otherwise $\left(\phi^{(l)}\right)^m(\xi_0)=0$ owing to the fact that zeros of $\phi$ have multiplicity at least $l+1.$ 

Since $\left(\phi^{(l)}\right)^m\not\equiv a_2,$ by Hurwitz's Theorem, there exists $\xi_n\longrightarrow\xi_0$ such that for sufficiently large $n,$ $$\left(f^{(l)}\right)^m(z_n+\rho_n\hat{\zeta_n}+\rho_n r_n\xi_n)=a_2~(\neq 0).$$ Also, $\phi_n(\xi_n)\longrightarrow\phi(\xi_0)=a_2^*~(\neq 0)$ implies that $f(z_n+\rho_n\hat{\zeta_n}+\rho_n r_n\xi_n)\longrightarrow a_1.$ However, $f(z_n+\rho_n\hat{\zeta_n}+\rho_n r_n\xi_n)\neq a_1.$ Indeed, if $f(z_n+\rho_n\hat{\zeta_n}+\rho_n r_n\xi_n)= a_1,$ then by hypothesis, $$f^{(i)}(z_n+\rho_n\hat{\zeta_n}+\rho_n r_n\xi_n)=0~ \mbox{for}~ 1\leq i\leq l$$ and hence $$\left(f^{(l)}\right)^m(z_n+\rho_n\hat{\zeta_n}+\rho_n r_n\xi_n)=0,$$ a contradiction to the fact that $\left(f^{(l)}\right)^m(z_n+\rho_n\hat{\zeta_n}+\rho_n r_n\xi_n)=a_2\neq 0.$ 

Thus $\left\{f(z_n+\rho_n\hat{\zeta_n}+\rho_n r_n\xi_n)\right\}$ is an infinite set and cannot be contained in the finite set $S_2,$ a contradiction.

\medskip

\textit{Case 2.} $h$ omits zero. Suppose that there exists $\zeta_0\in\mathbb{C}$ such that $h(\zeta_0)=b_1-a_1.$ Since $h\not\equiv b_1-a_1,$ by Hurwitz's Theorem, there exists $\zeta_n\longrightarrow\zeta_0$ such that for sufficiently large $n,$ $h_n(\zeta_n)=b_1-a_1.$ This implies that $$f(z_n+\rho_n\zeta_n)=b_1\in S_2.$$ By hypothesis, we have $$\left(f^{(l)}\right)^m(z_n+\rho_n\zeta_n)\in S_1.$$ Let  $$\left(f^{(l)}\right)^m(z_n+\rho_n\zeta_n)=a_3.$$ Then
\begin{eqnarray*} 
\left(h^{(l)}\right)^m(\zeta_0)=\lim\limits_{n\to\infty}\left(h_n^{(l)}\right)^m(\zeta_n)&=&\lim\limits_{n\to\infty}\rho_n^{lm}\left(f^{(l)}\right)^m(z_n+\rho_n\zeta_n)\\
&=&\lim\limits_{n\to\infty}\rho_n^{lm}~a_3=0.
\end{eqnarray*}
This shows that $$h(\zeta)=b_1-a_1\Rightarrow h^{(l)}(\zeta)=0.$$ Similarly, it can be shown that $$h(\zeta)=b_i-a_1\Rightarrow h^{(l)}(\zeta)=0,~i=2,3.$$
Now applying Lemma \ref{lem:7} to $h,$ we obtain 
\begin{eqnarray*}
3~T(r,h) &<& \left(l+1\right)\overline{N}(r,h)+ \overline{N}\left(r,\frac{1}{h}\right) + N\left(r,\frac{1}{h}\right)+S(r,h)\\
&=& \overline{N}(r,h)+l\overline{N}(r,h)+S(r,h)\\
&\leq& \left(\frac{1}{l+1}\right)N(r,h)+ N(r,h) + S(r,h)\\
&\leq& \left(\frac{1}{l+1}+1\right)T(r,h)+S(r,h).
\end{eqnarray*}
 which is a contradiction. Note that the case when $S_1=S_2$ can occur. If this is the case then, $$2~T(r,h)<\left(\frac{1}{l+1}+1\right)T(r,h)+S(r,h),$$ which is again a contradiction. Hence $f$ is a normal function.

\medskip

$(ii)$ If $a\notin S_1,$ then by the same argument as in {\it Case 2} of part $(i)$, one can easily deduce that $$h(\zeta)=a_i-a\Rightarrow h^{(l)}(\zeta)=0,~i=1,2,3.$$ Again, applying Lemma \ref{lem:7} to $h,$ we get 
\begin{eqnarray*}
3~T(r,h) &<& \left(l+1\right)\overline{N}(r,h)+ \overline{N}\left(r,\frac{1}{h}\right) + N\left(r,\frac{1}{h}\right)+S(r,h)\\
&\leq& \overline{N}(r,h)+l\overline{N}(r,h)+ \left(\frac{1}{l+2}\right)N\left(r,\frac{1}{h}\right)+N\left(r,\frac{1}{h}\right)+S(r,h)\\
&\leq& \left(\frac{1}{l+2}\right)N(r,h) +N(r,h)+ \left(\frac{1}{l+2}\right)N\left(r,\frac{1}{h}\right)+ N\left(r,\frac{1}{h}\right) + S(r,h)\\
&\leq& \left(\frac{2}{l+2}+2\right)T(r,h)+S(r,h).
\end{eqnarray*}
which is a contradiction.
Hence $f$ is a normal function.
\end{proof}

\section{Statements and Declaration}
{\bf Funding:} The authors declare that no funds, grants, or other support were received during the preparation of this manuscript.\\
\\
{\bf Competing Interests:} The authors declare that they have no conflict of interests. The authors have no relevant financial or non-financial interests to disclose.\\
\\
{\bf Author Contributions:} All authors contributed equally to this work.\\
\\
{\bf Data availability:} Data sharing not applicable to this article as no datasets were generated or analysed during the current study.

\end{document}